\documentclass[12pt]{article}
\usepackage{amssymb}
\usepackage{amsmath,amsthm}
\usepackage[latin1]{inputenc}
\usepackage{color}
\usepackage{graphicx}
 \usepackage{soul}
\usepackage[T1]{fontenc}

 \newtheorem{remark}{Remark}

 \newtheorem{proposition}[remark]{Proposition}
 \newtheorem{corollary}[remark]{Corollary}

\newcommand{\suchthat}{:}

\title{Alliances and related parameters in graphs}

\author{H.~Fernau\footnote{\small e-mail:\mbox{\tt
fernau\@@uni-trier.de}}  $^1$ and
J.~A. Rodr\'{\i}guez-Vel\'{a}zquez \footnote{\small e-mail:\mbox{\tt juanalberto.rodriguez\@@urv.cat}}
$^2$
\\
$^1$
{\small FB 4-Abteilung Informatikwissenschaften
}\\
{\small Universit\"{a}t Trier, 54286
Trier, Germany.}
\\
$^2${\small Departament d'Enginyeria Inform\`{a}tica i Matem\`{a}tiques
}\\
{\small Universitat Rovira i Virgili,  Av. Pa\"{\i}sos Catalans 26, 43007
Tarragona, Spain.}
}

\begin{document}

\maketitle

\begin{abstract}
 In this paper, we show that several graph parameters are known in different areas under completely different names.
More specifically, our observations connect signed domination, monopolies, $\alpha$-domination, $\alpha$-independence,
positive influence domination,
and a parameter associated to fast information propagation
in networks to parameters related to various notions of global $r$-alliances in graphs.
We also propose a new framework, called (global) $(D,O)$-alliances, not only in order to characterize
various known variants of alliance and domination parameters, but also to suggest a unifying framework for the study of alliances and domination.
Finally, we also give a survey on the mentioned graph parameters, indicating how results transfer due to our observations.
\end{abstract}

\noindent
\underline{Keywords}: graph parameters;
domination in graphs;
alliances in graphs;
computational complexity of graph problems

\section{Introduction to  graphs and  alliances}

Let $G=(V,E)$ denote a simple graph.
For a non-empty subset $S\subseteq V$, and a vertex $v\in V$, we
denote by $N_S(v)$ the set of neighbors $v$ has in $S$.
We denote the degree of $v$ in $S$ by $\delta_S(v)=|N_S(v)|.$
We write $N_S[v]=N_S(v)\cup\{v\}$ to denote the closed neighborhood.
For vertex sets $U$ and $S$, $N_S(U)=\bigcup_{u\in U}N_S(u)$, and $N_S[U]=N_S(U)\cup U$.
If $S=V$, we suppress the subscript $V$ in the notations introduced so far.
$\overline{S}$ denotes the complement of $S$, i.e., $\overline{S}=V\setminus S$.

 Recall that $D\subseteq V$ is a \emph{dominating set} in $G=(V,E)$ if $N[D]=V$~\cite{HHS98}.
We consider several variants of dominating sets.
 Consider the following condition 
 \begin{equation}\label{condition-alliance}
 \delta_S (v) \geq \delta_{\bar S} (v) + r,\tag{*}
 \end{equation}
  which states
that a vertex $v$ has at least $r$ more neighbors in $S$ than it has in $\bar S$. A dominating set  $S \subseteq V$ that
satisfies Condition $(\ref{condition-alliance})$ for every vertex $v \in S$ is called a \emph{global defensive $r$-alliance};
if $S\neq\emptyset$
satisfies Condition $(\ref{condition-alliance})$ for every vertex $v \in \overline{S}$, then
$S$ is called  \emph{global offensive $r$-alliance}.

A set $S\subseteq V$ is a \emph{global  powerful $r$-alliance} if $S$ is both a global defensive $r$-alliance and a global
offensive $(r+2)$-alliance, \cite{BerRodSig2011,Brietal2009,Sig2007}. Global  powerful $0$-alliances
are also known as strong powerful alliances or strong dual alliances.

Alliances were introduced in several papers between 2000 and 2010 and were studied in various PhD theses and many papers.
In order not to overdo, we only list the first papers and theses now:
\cite{Enc2009,Favetal2002a,KriHedHed2004,Sha2001,Sig2007}.
It is rather well-known that some of the concepts of alliances were invented independently and under different names.
For instance, defensive alliances appear also in \cite{FlaLawGil2000}.
Also, it is observed in~\cite{KriHedHed2004} that signed dominating functions induce global strong powerful alliances,
or in the terminology introduced above, global powerful $0$-alliances.  However, no proof of this fact appeared, and our Proposition~\ref{prop-signed-dom} even provides a general characterization of
signed domination in terms of alliances.
We will exhibit in this paper several connections between global alliances of various types
and other parameters that were introduced in the literature of graphs and networks.
Seeing these connections should be helpful for researchers both in domination and in alliance theory.

\section{$(D,O)$-alliances}

In order to properly characterize various situations, we suggest the following generalization of alliance parameters.
 A \emph{$(D,O)$-alliance}, with $D,O\subseteq\mathbb{Z}$ in a graph $G=(V,E)$ is a vertex set $S$ with
\begin{enumerate}
\item $\forall v\in S$: $\delta_S(v)-\delta_{\overline{S}}(v)\in D$ and
\item  $\forall v\in N(S)\setminus S$: $\delta_S(v)-\delta_{\overline{S}}(v)\in O$.
\end{enumerate}
Hence, defensive $r$-alliances can be addressed as $(\{d\in\mathbb{Z}\suchthat d\geq r\},\mathbb{Z})$-alliances,
and offensive  $r$-alliances can be written as $(\mathbb{Z},\{o\in\mathbb{Z}\suchthat o\geq r\})$-alliances.
Likewise, $(\{d\in\mathbb{Z}\suchthat d\geq r\},\{o\in\mathbb{Z}\suchthat o\geq r+2\})$-alliances are also known as powerful $r$-alliances.
A $(D,O)$-alliance which is at the same time a dominating set is called \emph{global}.

Recall and compare this with the following defnition that is very similar to J.~A. Telle's proposal~\cite{Tel94,Telle94}:
A \emph{$[\sigma,\rho]$-set}, with $\sigma,\rho\subseteq\mathbb{N}$, in a graph $G=(V,E)$ is a vertex set $S$ with
\begin{enumerate}
\item $\forall v\in S$: $\delta_S(v)\in\sigma$ and
\item $\forall v\notin S$: $\delta_S(v)\in\rho$.
\end{enumerate}

%
%
The framework of $(\sigma,\rho)$-domination has triggered quite some research on different number sets prescribing different forms of domination and independence.
The framework that we propose might serve for a similar purpose. 

Let us remark that it is possible to model other forms of alliances introduced in the literature within this new framework.
For instance, I.~G. Yero introduced \emph{boundary alliances} in his PhD thesis~\cite{Yer2010}.
These can be easily introduced with our terminology as follows:
\begin{itemize}
 \item A $(\{r\},\mathbb{Z})$-alliance is called a \emph{boundary defensive $r$-alliance};

\item a $(\mathbb{Z},\{r\})$-alliance is called a \emph{boundary offensive $r$-alliance};

\item an $(\{r\},\{r+2\})$-alliance is called a \emph{boundary powerful $r$-alliance}.
\end{itemize}

Combinatorial results on (global) boundary defensive $r$-alliances and (global) boundary powerful $k$-alliances are also contained in~\cite{YerRod2010,YerRod2013}, respectively.
\begin{remark}
Note that $S\subseteq V(G)$ is a global $(\{r\},\mathbb{Z})$-alliance if and only if $\overline{S}$ is a global $(\mathbb{Z},\{-r\})$-alliance.
\end{remark}

Let us comment on regular graphs in the following.
If a graph $G$ is $r$-regular, i.e., if $\delta(v)=r$ for all vertices $v\in V(G)$, then we can restrict ourselves
to $[\sigma,\rho]$-sets with $\sigma,\rho\subseteq\{0,\dots,r\}$ and to $(D,O)$-alliances with $D,O\subseteq \{-r,\dots,r\}$.
Then, we can observe:
\begin{proposition}
 Let $G$ be an $r$-regular graph and $\sigma,\rho\subseteq\{0,\dots,r\}$,
Then, $S\subseteq V(G)$ is a $[\sigma,\rho]$-set if and only if $S$ is a $(D,O)$-alliance, where
$D=\{d\in\mathbb{Z}\suchthat \frac{d+r}{2}\in\sigma\}$
and
$O=\{o\in\mathbb{Z}\suchthat \frac{o+r}{2}\in\rho\}$
\end{proposition}

\begin{proof}
 For every $v\in V(G)$ we have $\delta_S(v)-\delta_{\overline{S}}(v)=\delta_S(v)-(r-\delta_S(v))=2\delta_S(v)-r$. So, for $v\in S$ and $d(v)=2\delta_S(v)-r$  it follows $\delta_S(v)\in \sigma$ if and only if $d(v)\in D$.
A similar computation applies to $v\notin S$.
\end{proof}

For instance, so-called total perfect dominating sets are described as $[\{1\},\{1\}]$-sets.
On cubic graphs, this corresponds to $(\{-1\},\{-1\})$-alliances.
Similarly, so-called perfect codes, also known as efficient dominating sets, can be described as $[\{0\},\{1\}]$-sets.
On cubic graphs, this corresponds to $(\{-3\},\{-1\})$-alliances, or in other words,
to boundary powerful $(-3)$-alli\-an\-ces.
The known NP-hardness results for total perfect domination and for perfect codes on cubic graphs
(see \cite{Telle94}) immediately translate to NP-hardness results on boundary powerful $(-3)$-alli\-ances, for instance.
This already shows some of the possible connections between the theory of (global) alliances and that of variants of domination.

\section{Signed domination}

Let $f : V \to\{-1, +1\}$ be a function which assigns each vertex of a graph $G=(V,E)$ a sign.
 Then, $f$ is called \emph{signed dominating function}
if for every $v \in V$, $f (N [v]) \geq1$, \cite{Dunetal95,HHS98} (here
and in the sequel, we use the notation $f(S)=\sum_{v\in S}
f(v)$ for a subset $S\subset V$). 
 Given an integer $k$,  we also consider the notion of \emph{signed $k$-dominating function} where for every $v\in V$, $f (N [v]) \geq k$, as introduced by C. Wang in~\cite{Wan2012}. The \emph{signed $k$-dominating set} associated to $f$ is  the set of vertices with value $+1$ assigned
by $f$.
Clearly, a signed $k$-dominating set is a dominating set if $k\geq 1$.
We connect signed $k$-dominating sets with global powerful $(k-1)$-alliances by the following proposition.

\begin{proposition}\label{prop-signed-dom}
Let $G$ be a graph and let $k\ge 1$ be an integer. A set  $S\subseteq V(G)$ is a  signed $k$-dominating set  if and only if $S$   is a global
$(\{d\in\mathbb{Z}\suchthat d\geq k-1\},\{o\in\mathbb{Z}\suchthat o\geq k+1\})$-alliance.
\end{proposition}

\begin{proof}
Let  $g : V(G) \to\{-1, +1\}$ be a function and let $S\subseteq V(G)$ composed of the vertices of $G$ with value $+1$ assigned by  $g$.
 So, for every vertex $v\in \overline{S}$,
$$\sum_{u\in N[v]}g(u)=-1+\delta_S(v)-\delta_{\overline{S}}(v)$$
and for every vertex $v\in S$,
$$\sum_{u\in N[v]}g(u)=1+\delta_S(v)-\delta_{\overline{S}}(v).$$
Therefore, $S$ is a global
$(\{d\in\mathbb{Z}\suchthat d\geq k-1\},\{o\in\mathbb{Z}\suchthat o\geq k+1\})$-alliance if and only if  $S$ is a  signed $k$-dominating set.
\end{proof}

%

\subsection*{Consequences of Proposition~\ref{prop-signed-dom}}

Shafique \cite{Sha2001} obtained that the question if there exists a  powerful $0$-alliance of size at most $\ell$ is NP-complete.
Due to Proposition~\ref{prop-signed-dom}, the same result has been already shown in \cite{HatHenSla95}, in the terminology of signed domination.
Generalized NP-hardness results towards signed $k$-domination are contained in~\cite{Lia2013}.
But here, the literature on alliances was quicker, as it was shown in~\cite{FerRodSig08} that
the question of finding a global powerful $k$-alliance of size at most $\ell$ is NP-complete for any integer~$k$.
It has been noticed in \cite{FerRai07} that this question is fixed-parameter tractable in general graphs.
For signed domination, parameterized complexity results are collected in~\cite{Zheetal2012}.
In a sense, several of the aforementioned results are sharpened there, for instance, it is shown that the corresponding decision problem is NP-complete even on  bipartite or on chordal graphs.
Moreover, while the quadratic kernel of \cite{FerRai07} coincides with the result from \cite{Zheetal2012}, including the main reduction rule,
Y. Zheng \emph{et al.} also obtain a small linear kernel for
\textsc{Signed Domination}
on planar graphs.

Notice that while with the following, more network-oriented notions, most research has been undertaken from the viewpoint of complexity, this is different with signed domination.
Hence, also those graph theorists working in the combinatorics of alliances may profit from the following papers, due to Proposition \ref{prop-signed-dom}: 
\cite{CheSon2008,Dunetal95,FurMub99,HaaWex2002,HaaWex2004,HHS98,Mat2000,Pog2010,PogZve2010,ShaCheKan2008,ShiKanWu2010,Ung96,Zhaetal99}.

\subsection*{Three variants of signed domination}


\paragraph{Signed total domination.}
A \emph{signed total dominating function} of $G=(V,E)$ is defined in~\cite{Zel2001}
 as a function $f: V \rightarrow  \{-1, 1\}$ satisfying $f(N(v))\ge 1$ for all $v\in V$.
The \emph{signed total dominating set} associated to $f$ is  the set of vertices with value $+1$ assigned.
This was again generalized by C. Wang~\cite{Wan2012} towards signed total $k$-domination, based
on a function $f: V \rightarrow  \{-1, 1\}$ satisfying $f(N(v))\ge k$ for all $v\in V$.
By analogy to the proof of Proposition \ref{prop-signed-dom},
we can see the following result.
\begin{proposition}
\label{prop-signed-total}
Let $G$ be a graph and let $k\ge 1$ be an integer. A set  $S\subseteq V(G)$ is a  signed  total  $k$-dominating set
 if and only if $S$ is a  global
 $(\{d\in\mathbb{Z}\suchthat d\geq k\},\{o\in\mathbb{Z}\suchthat o\geq k\})$-alliance.
\end{proposition}
Combinatorial results on signed total dominating functions can be found in
\cite{CheSon2008,Har2003,HarHat2004,Zel2001}.
For combinatorial results on defensive (offensive) $k$-alliances we cite,
for instance,  \cite{Cheetal2009,FerRodSig09,RodYerSig2009,RodSig2009,Sha2001,Sig2007,Yer2010}.
Likewise, complexity results have been obtained in this setting; we only refer to \cite{Lee2009} and the literature quoted therein
for signed total domination and to \cite{Lia2013} for signed  total  $k$-domination.

\paragraph{Minus domination: introducing neutral elements.}
It might be also interesting to observe that there is also another concept related to signed domination, namely, minus domination~\cite{Dunetal99,HHS98}.
With the idea of alliances in the back of your mind, this can be interpreted as partitioning the vertices not only into friends (allies) and enemies,
but also allowing neutral vertices.
A possible definition could ask for a function $f:V\to\{-1,0,1\}$, the \emph{minus dominating function}, for the graph $G=(V,E)$ such that,
 for every $v \in V$, $f (N [v]) \geq1$. The \emph{minus dominating set} would collect $D=\{v\in V\suchthat  f(v)=1\}$.
 Equivalently, we may be allowed to delete the neutral elements $N=\{v\in V\suchthat f(v)=0\}$, as $D$ is a signed domination set and hence a global powerful $0$-alliance
in $G-N$.
 For complexity aspects of signed and minus domination (in particular on degree-bounded graphs), we refer to \cite{Dam2001,Dunetal96a} and the literature quoted therein.
Neutral elements were considered in the context of alliances in~\cite{RadRez2010}.
In this spirit, we suggest the following notion:
A dominating set $D$ in $G$ is a \emph{global $(D,O)$-alliance with neutrals $N$}, $N\subseteq V(G)$,
if it is a global $(D,O)$-alliance in $G-N$. Thus, in analogy to Proposition~\ref{prop-signed-dom}, we establish the following result.
\begin{proposition}
Let $G$ be a graph. A set  $S\subseteq V(G)$ is a minus dominating set
 if and only if  there exists a vertex set $N$, $N\cap S=\emptyset$, such that $S$ is a global  $(\{d\in\mathbb{Z}\suchthat d\geq 0\},\{o\in\mathbb{Z}\suchthat o\geq 2\})$-alliance with neutrals~$N$.
\end{proposition}

In the spirit of \cite{DelWan2010,Wan2012}, it might be interesting to consider minus $k$-dominating sets.
These would correspond to
global  $(\{d\in\mathbb{Z}\suchthat d\geq k-1\},\{o\in\mathbb{Z}\suchthat o\geq k+1\})$-alliances with neutrals~$N$, as long as $k\geq 1$.

\paragraph{Signed efficient domination.} Another issue is \emph{efficient domination}~\cite{BanBarSla88,HHS98}. This was also considered in relation with signed and minus domination \cite{Banetal96,LuPenTan2003}
and could be also formalized within the framework of global $(D,O)$-alliances.

Recall that a signed dominating function  $f$ is called \emph{signed efficient dominating function}
if for every $v \in V$, $f (N [v]) =1$.  In analogy to Proposition~\ref{prop-signed-dom}, we can state the following result, which
gives a link to global boundary powerful $0$-alliances.
\begin{proposition}
Let $G$ be a graph.
A set  $S\subseteq V(G)$ is an efficient  signed dominating set
if and only if $S$ is a global $(\{0\},\{2\})$-alliance.
\end{proposition}
Interestingly, the question of the existence of such an alliance is just as hard as the question of finding a smallest one,
as all efficient signed domination functions have the same weight, as shown by D.~W. Bange \emph{et al.} in~\cite{Banetal96}.

Those interested in generalizing signed efficient domination towards signed efficient $k$-domination, which
should correspond to a function $f$ with $f (N [v]) =k$ for  every $v \in V$, should bear in mind that this notion would
coincide with that of a global boundary powerful $k$-alliance.

\section{Monopolies}

From \cite{MisRadSiv2002}, we learn the following notions.
A \emph{partial monopoly} in a graph $G=(V,E)$ is a vertex set $X$ such that for all $v\notin X$,
\begin{equation}\label{condition-monopoly}
|N[v]\cap X|\geq \frac{1}{2}|N[v]|.\tag{**}
\end{equation} 
A \emph{monopoly}, as defined in \cite{MisRadSiv2002}, satisfies (\ref{condition-monopoly}) for all $v\in V$, not only for those $v$ in the complement of $X$.
In the terminology introduced by D. Peleg in~\cite{Pel2002}, a partial monopoly is a \emph{self-ignoring 1-monopoly} and a monopoly is a \emph{1-monopoly}.
This notion of partial monopoly was  introduced in \cite{KhoNSZ2013} under the name \textit{strict monopoly}.
We first provide a characterization in terms of global offensive 1-alliances.

\begin{proposition}\label{prop-partial-monopoly}
 Let $G$ be a graph. A set  $X\subseteq V(G)$  is a partial monopoly   if and only if $X$ is a global  $(\{\mathbb{Z},\{o\in\mathbb{Z}\suchthat o\geq 1\})$-alliance.
\end{proposition}

\begin{proof}
If $X\subseteq V(G)$, then for every $v\in \overline{X}$ the following expressions are equivalent:
\begin{align*}
|X\cap N[v]|&\ge \frac{1}{2}|N[v]|\\
\delta_X(v)&\ge \frac{1}{2} \left(\delta_X(v)+\delta_{\overline{X}}(v) + 1\right)\\
\frac{1}{2}\delta_X(v)&\ge \frac{1}{2}\delta_{\overline{X}}(v) + \frac{1}{2}\\
\delta_X(v)-\delta_{\overline{X}}(v)&\ge 1.
\end{align*}
So, $X$ is a partial monopoly if and only if $X$ is a global $(\mathbb{Z},\{o\in\mathbb{Z}\suchthat o\geq 1\})$-alliance.
\end{proof}

\begin{proposition}\label{prop-monopoly}
 Let $G$ be a graph. A set  $X\subseteq V(G)$ is a  monopoly
if and only if $X$ is a global
$(\{d\in\mathbb{Z}\suchthat d\geq -2\},\{o\in\mathbb{Z}\suchthat o\geq 1\})$-alliance.
\end{proposition}
\begin{proof}
If $X$ is a monopoly, then for every $v\in X$ we have
\begin{align*}
|X\cap N[v]|&\ge \frac{1}{2}|N[v]|\\
\delta_X(v)+1&\ge \frac{1}{2} \left(\delta_X(v)+\delta_{\overline{X}}(v) \right)\\
\frac{1}{2}\delta_X(v)+1&\ge \frac{1}{2}\delta_{\overline{X}}(v) \\
\delta_X(v)-\delta_{\overline{X}}(v)&  \ge-2.
\end{align*}
So, if $X$ is a monopoly, then $X$ is a $(\{d\in\mathbb{Z}\suchthat d\geq -2\},\mathbb{Z})$-alliance, and vice versa.  Hence, by Proposition
\ref{prop-partial-monopoly}, we conclude the proof.
\end{proof}

D. Peleg actually introduced more general notions. Recall that the $r$th power of a graph $G $ is a graph with the same set of vertices as $G$
and an edge between two vertices if and only if there is a path of length at most $r$ between them.
 $X$ is a \emph{self-ignoring $r$-monopoly} in $G=(V,E)$ if and only if $X$ is a partial monopoly in the $r$th graph power of $G$.
$X$ is an \emph{$r$-monopoly} in $G=(V,E)$ if and only if $X$ is a 1-monopoly in the $r$th graph power of $G$.
This might motivate to study alliances in graph powers.

Apart from this, reference~\cite{Pel2002} (as an overview) and then \cite{Beretal2003,Linetal93} contain quite some combinatorial results on (variants of) monopolies
that would translate to properties of global alliances by virtue of the above propositions (or by analogy).
Also, in \cite{Beretal2003}  relations to signed domination are informally stated; our paper can be seen as formalizing this intuition.

It is stated in~\cite{Pel2002} without proof that finding monopolies (in certain variants) of size at most $k$ is NP-hard.
The inapproximability results from~\cite{MisRadSiv2002} imply even stronger results for finding smallest (partial) monopolies.
A second generalization can be found in \cite[Theorem~4]{FerRodSig09}, where it is shown that, for each $r$, deciding if there exists a global offensive $r$-alliance
of size at most $k$ in a a graph is NP-complete.

 S.~Mishra \emph{et al.}~\cite{MisRadSiv2002,MisRao2006} show that (partial) monopolies no bigger than $c$ times the minimum can
be found in polynomial time in cubic graphs, but no PTAS exists on graphs of bounded degree.
 S.~Mishra \emph{et al.}~\cite{MisRadSiv2002} also exhibit interesting combinatorial bounds for  (partial) monopolies in graphs of bounded degree
that should be interesting to those working in the theory of alliances.
Likewise, H. Fernau and D. Raible have shown in~\cite{FerRai07} that alliance problems are fixed-parameter
tractable, which translates to according statements for (partial) monopolies.

\section{$\alpha$-domination and $\alpha$-independence}

A kind of open-neighborhood variant of
partial monopolies have been generalized towards so-called $\alpha$-dominating sets in~\cite{Dunetal2000}.
More precisely, for any $0<\alpha\leq 1$, an \emph{$\alpha$-dominating set}  in a graph $G=(V,E)$ is a vertex set $X$ such that for all $v\notin X$, $|N(v)\cap X|\geq
 \alpha|N(v)|$. Obviously, a $\frac{1}{2}$-dominating set is an open-neighborhood variant of  a partial monopoly.
It has been observed in~\cite{KriHedHed2004}  that, for any $\alpha>\frac{1}{2}$, an {$\alpha$-dominating set} $X$ has more neighbors in $X$
than it has in $\overline{X}$, turning it into a global $(\mathbb{Z},\{o\in\mathbb{Z}\suchthat o\geq 1\})$-alliance.
Unfortunately, this does not yield a characterization of $\alpha$-domination for  $\alpha>\frac{1}{2}$.
For instance, consider $G=K_{2r}$. In this case for $X\subset V(K_{2r})$ such that $|X|=r$, one has $\delta_X(v)=r=\delta_{\overline{X}}(v)+1$ for every $x\in \overline{X}$.
 So, $X$ is a global $(\mathbb{Z},\{o\in\mathbb{Z}\suchthat o\geq 1\})$-alliance. But, $r= |N(v)\cap X|\geq
 \alpha|N(v)|=\alpha (2r-1)$ implies that $\alpha \le \frac{r}{2r-1}\rightarrow \frac{1}{2}$, when $r \rightarrow \infty$.


It is also claimed in~\cite{KriHedHed2004} that, if $\alpha\leq \frac{1}{2}$, then the complement $\overline{X}$ of  an {$\alpha$-dominating set} $X$
forms a strong defensive alliance. This statement is surely not true, as the condition that at least a certain fraction of neighbors of $v\in\overline{X}$
is in $X$ does say that at least a certain fraction of neighbors is in $\overline{X}$, which should be true if the claim were true.
However,
for $\alpha=\frac{1}{2}$, we can replace this statement by the following one which relates $\frac{1}{2}$-domination and global offensive $0$-alliances.

\begin{proposition}\label{prop-alpha2}
 Let $G$ be a graph. A set  $X\subseteq V(G)$ is a {$\frac{1}{2}$-dominating set}  if and only if $X$ is a global $(\mathbb{Z},\{o\in\mathbb{Z}\suchthat o\geq 0\})$-alliance.
\end{proposition}

\begin{proof}
Consider some {$\frac{1}{2}$-dominating set} $X\subseteq V(G)$, i.e.,
\begin{align*}
| N(v)\cap X|&\ge \frac{1}{2}|N(v)|\\
\delta_X(v)&\ge \frac{1}{2} \left(\delta_X(v)+\delta_{\overline{X}}(v) \right)\\
\delta_X(v)&\ge \delta_{\overline{X}}(v).
\end{align*}
Hence, $X$ is a $(\mathbb{Z},\{o\in\mathbb{Z}\suchthat o\geq 0\})$-alliance.
The converse follows similarly.
\end{proof}

From results in~\cite{Dunetal2000}, we can deduce that the problem(s) of determining if there exists a global offensive $0$- or $1$-alliance in a graph is NP-complete
even on cubic graphs, complementing the complexity results mentioned above.
Combinatorial results on $\alpha$-domination are rare; \cite{DahRauVol2008,Dunetal2000} seems to be a complete list today.
Let us mention in passing that F. Dahme, D. Rautenbach and L. Volkmann~\cite{DahRauVol2008} proposed another concept:
an \emph{$\alpha$-independent set}  in a graph $G=(V,E)$ is a vertex set $X$ such that for all $v\in X$, $|N(v)\cap X|\leq
 \alpha|N(v)|$.
In other words, for all $v\in X$, $\delta_X(v)\leq \alpha(\delta_X(v)+\delta_{\overline{X}}(v))$.
For $\alpha=\frac{1}{2}$, this means that $\delta_X(v)\leq \delta_{\overline{X}}(v)$.
Hence, $\overline{X}$ forms a global $(\mathbb{Z},\{o\in\mathbb{Z}\suchthat o\geq 0\})$-alliance.
The reverse direction can be similarly seen.

\begin{proposition}
 Let $G$ be a graph. A set  $X\subseteq V(G)$ is a is a {$\frac{1}{2}$-independent set}
 if and only if $\overline{X}$ is a global $(\mathbb{Z},\{o\in\mathbb{Z}\suchthat o\geq 0\})$-alliance.
\end{proposition}

This also provides a Gallai-type result that seems to be a new observation; such results are also known as complimentarity results; see~\cite[Sec. 10.4]{HHS98}.

\begin{corollary}
 The order of a graph equals the sum of the size of the smallest {$\frac{1}{2}$-dominating set} plus the size of the largest  {$\frac{1}{2}$-independent set}.
\end{corollary}

F. Cicalese, M. Milanic and U. Vaccaro introduce in~\cite{CicMilVac2011} the notion of  \emph{(total)
$q$-domination}: $S\subseteq V$ is a $q$-dominating set in the graph $G=(V,E)$ if,  for every $v \in V \setminus S$, it
holds that $|N(v) \cap S| > q|N(v)|$; it is called total if this property also holds for  $v\in S$.
We refrain from stating analogues to the previous propositions, as these statements might look confusing again, in particular when specializing to the case of $q=\frac{1}{2}$.
Relations to monopolies and fast information propagation (see next section) are also (informally) observed  in that paper.
The in approximability results from~\cite{CicMilVac2011} would be also interesting in and transfer to some other contexts (for instance, in the theory of alliances).

\section{Positive influence domination and fast information propagation}

Another related notion is that  of a \emph{positive influence dominating set}; see \cite{Dinetal2013,DinNguTha2013,Wanetal2011,Zhaetal2012}.
This is a  vertex set $X \subseteq V$ such that any $v \in V$ is neighbor of
at least $\left\lceil \frac{\delta(v )}{2} \right\rceil$ many vertices from $X$; recall that  $\delta(v)=|N(v)|$ denotes the degree of vertex $v$.


\begin{proposition}\label{prop-positive-influence}
 Let $G$ be a graph. A set  $X\subseteq V(G)$ is a  positive influence dominating set
 if and only if $X$ is a global $(\{d\in\mathbb{Z}\suchthat d\geq 0\},\{o\in\mathbb{Z}\suchthat o\geq 0\})$-alliance.
\end{proposition}
\begin{proof}
Let $X\subseteq V(G)$. For every $v\in V(G)$ the following expressions are equivalent.
\begin{align*}
\delta_X(v)&\ge \left\lceil \frac{\delta(v )}{2} \right\rceil \\
\delta_X(v)&\ge  \frac{\delta(v )}{2}  \\
\delta_X(v)&\ge  \frac{\delta_X(v )+\delta_{\overline{X}}(v)}{2}  \\
\delta_X(v)&\ge  \delta_{\overline{X}}(v).  \\
\end{align*}
Therefore, $X$ is a positive influence dominating set
 if and only if $X$ is a global $(\{d\in\mathbb{Z}\suchthat d\geq 0\},\{o\in\mathbb{Z}\suchthat o\geq 0\})$-alliance.
\end{proof}

For this graph parameter, mostly (in)approximability results have been derived.
Notice that Proposition~\ref{prop-positive-influence}
also yields a characterization of positive influence domination in terms of
signed total $0$-domination for graphs 
without isolated vertices, along the lines of  Proposition~\ref{prop-signed-total}.
Namely observe that the only problem with that equivalence is due to isolated vertices,  which could be assigned $-1$,
and yet the (empty!) sum of all $f$-values of all their neighbors would be zero.
Conversely, if the graph contains no isolated vertices, then every vertex has at least one neighbor, so that
a signed total $0$-dominating set is in particular a  dominating set and can be hence characterized as a global $(\{d\in\mathbb{Z}\suchthat d\geq 0\},\{o\in\mathbb{Z}\suchthat o\geq 0\})$-alliance.


 In line with usual notations in domination theory, positive influence dominating sets are called
total positive influence dominating sets in~\cite{DinNguTha2013}, while they refer to sets $X$ where every $v \in \overline{X}$ is neighbor of
at least $\left\lceil \frac{\delta(v )}{2} \right\rceil$ many vertices from $X$ as positive influence dominating sets.
Clearly, such sets can be addressed as  global $(\mathbb{Z},\{o\in\mathbb{Z}\suchthat o\geq 0\})$-alliances or as global offensive $0$-alliances and hence correspond to $\frac{1}{2}$-dominating sets
 according to
Proposition~\ref{prop-alpha2}. We discuss similar notions in the following.


Positive influence dominating sets are closely related to problems that involve a diffusion process though a network.  Such problems share
a common idea of selecting an initial subset of vertices to activate in a graph
such that, according to a propagation rule, all vertices are activated once the
propagation process stops. One such representative problem is the \textsc{Target Set
Selection} problem first introduced in \cite{Che2009}.
For each $v \in V$, there is a
threshold value $t(v) \in \mathbb{N}$, where $1 \le t(v) \le \delta(v)$. Initially, the states of all vertices are
inactive. We pick a subset of vertices, the target set, and set their state to be active.
After that, in each discrete time step, the states of vertices are updated according to
the following rule: An inactive vertex $v$ becomes active if at least $t(v)$ of its neighbors
are active. The process runs until either all vertices are active or no additional vertices
can update states from inactive to active. The 
 following optimization problem, called \textsc{Target Set
Selection}, was considered in \cite{Che2009}: Which subset of vertices should be targeted at the beginning such that
all (or a fixed fraction of) vertices in the graph are active at the end?  The goal  considered in \cite{Che2009} was to minimize the size of the target set. 
For the  thresholds  $t(v)=\left\lceil \frac{\delta(v)}{2} \right\rceil$,  
%
this is tightly linked to the following model: 
According to F.~Zou \emph{et al.}~\cite{Zouetal2010},
given a fixed model of information diffusion in a network and some latency bound $d$,
the \textsc{Fast Information Propagation Problem} asks, given a graph $G=(V,E)$ and an integer $k$,
if there exists some set $P$ with at most $k$ vertices such that, after $d$ rounds, all vertices
have obtained the information originally (only) located at $P$.
That paper~\cite{Zouetal2010} mostly considered the Majority Theshold Model for information diffusion, again motivated by (the conference version of) \cite{Pel2002}.
More formally, let $\textrm{MAJ}(P)=\{v\in V\suchthat |N(v)\cap P|\geq \frac{1}{2}|N(v)|\}\cup P$. (At least, this is our interpretation of ``a node becomes active only when
at least half of its neighbors are active.'')
The operator MAJ can be iterated.
This means, $\textrm{MAJ}^1(P)=\textrm{MAJ}(P)$, and for $d>1$,  $\textrm{MAJ}^d(P)=\textrm{MAJ}^{d-1}(\textrm{MAJ}(P))$.
So, a \emph{$d$-MAJ set} is a vertex set $P$ such that $\textrm{MAJ}^d(P)=V$.
We relate a vertex set from which information can be spread to every vertex in just one round with global offensive $0$-alliances as follows.

\begin{proposition}\label{prop-propagation}
 Let $G$ be a graph. A set  $P\subseteq V(G)$ is a global  $(\mathbb{Z},\{o\in\mathbb{Z}\suchthat o\geq 0\})$-alliance if and only if $\textrm{MAJ}^1(P)=V$.
\end{proposition}
\begin{proof}
If $\textrm{MAJ}^1(P)=V$, then for every $v\in \overline{P}$ we have
\begin{align*}
| N(v)\cap P|&\ge \frac{1}{2}|N(v)|\\
\delta_P(v)&\ge \frac{1}{2} \left(\delta_P(v)+\delta_{\overline{P}}(v) \right)\\
\delta_P(v)&\ge \delta_{\overline{P}}(v).
\end{align*}
So, if  $\textrm{MAJ}^1(P)=V$, then $P$ is  a global $(\mathbb{Z},\{o\in\mathbb{Z}\suchthat o\geq 0\})$-alliance, and vice versa.
\end{proof}

Hence, the NP-hardness result proved in \cite{Zouetal2010} in a rather complicated manner also follows from \cite[Theorem~4]{FerRodSig09} or from \cite{Dunetal2000}.

C. Bazgan and
               M. Chopin~\cite{BazCho2012} introduced a kind of dual notion of a positive influence domination.
 A set $R\subseteq V(G)$ is a \emph{robust set with majority thresholds} if, for all $v\in V(G)$, $\delta_R(v)<\left\lceil \frac{\delta(v )}{2} \right\rceil$.
A robust set is called a harmless set in \cite{Cho2013}. 
\begin{proposition}\label{prop-robust}
  Let $G$ be a graph. A set $R\subseteq V(G)$ is a  $(\{o\in\mathbb{Z}\suchthat o<0\},\{o\in\mathbb{Z}\suchthat o<0\})$-alliance if and only if
$R$ is a robust set with majority thresholds.
\end{proposition}
\begin{proof}
Let $R\subseteq V(G)$. For every $v\in V(G)$ the following expressions are equivalent.
\begin{align*}
\delta_R(v)&< \left\lceil \frac{\delta(v )}{2} \right\rceil \\
\delta_R(v)&<  \frac{\delta(v )}{2}  \\
\delta_R(v)&<  \frac{\delta_R(v )+\delta_{\overline{R}}(v)}{2}  \\
\delta_R(v)&<  \delta_{\overline{R}}(v).  \\
\end{align*}
Therefore, $R$ is a robust set with majority thresholds
 if and only if $R$ is a  $(\{o\in\mathbb{Z}\suchthat o<0\},\{o\in\mathbb{Z}\suchthat o<0\})$-alliance.
\end{proof}
M. Chopin and his co-authors \cite{BazCho2012,BazCNS2013,Cho2013,ChoNNW2012} also considered parameterized complexity and approximability aspects of \textsc{Target Set Selection}
and related graph problems. Notice that the natural optimization problem related to robust sets is a maximization problem.

\section{Further interesting aspects}

The fact that the same graph parameters were obviously independently introduced in the literature
also bears some possibly fruitful ideas for future research, possibly beyond the idea of studying $(D,O)$-alliances for different number sets $D$ and $O$.
We only indicate some of these in the following.

\begin{itemize}
\item From the original motivation behind alliances, research on (global) $(\{d\in\mathbb{Z}\suchthat d\geq k\},\{o\in\mathbb{Z}\suchthat o\geq \ell\})$-alliances for two
given integers $k,\ell$ might be of particular interest. With $\ell=k+2$, we are back to the notion of (global) powerful $k$-alliances, but the proposed new notion is more flexible,
as someone who likes to apply this theory can scale in how much security in terms of defensive or offensive power (s)he is aiming at.
Observe that many of our characterization results can be interpreted in this terminology. In other words, for several $k,\ell$,
  global  $(\{d\in\mathbb{Z}\suchthat d\geq k\},\{o\in\mathbb{Z}\suchthat o\geq \ell\})$-alliances
have already been studied in the domination literature under different names. Our new proposed terminology might help bridge those results, presenting them in a uniform way.
 \item J.~H. Hattingh, M.~A. Henning and P.~J. Slater~\cite{HatHenSla95} also considered the parameter \emph{upper signed domination},
 which is the maximum cardinality of an inclusion-minimal signed domination set.
 This has been later generalized to upper signed $k$-domination in~\cite{DelWan2010}.
 This type of parameter is well-known from domination theory, but has received, to our knowledge, little
 attention so far in the world of alliance parameters, apart from~\cite{KriHedHed2004}.

 \item  Work on fast information propagation could stir interest in dynamic versions of the notion of alliances.
Notice that alliances on graph powers would formalize similar ideas, here we refer again to our section on monopolies.
In this context, it should be good to recall that dominating set in the $r$th power of a graph has been studied under the name of
\emph{distance-$r$ domination}; see \cite[Sec. 7.4]{HHS98}, based on a more general notion introduced in~\cite{Sla76}.
 \item Inspired by well-known work on edge domination, there is already a whole range of papers that consider \emph{signed edge domination},
 starting with \cite{Xu2001}.
As a manifestation of the diversity of such papers, we mention a few of them in the following: \cite{KarKhoShe2009,KarSheKho2008,KarSheKho2012,Karetal2013,ZhaXu2011}.
All these notions can be interpreted as $(D,O)$-edge alliances (for appropriate sets $D,O$). 
 However, little has been done so far on \emph{edge alliances}, apart from very few results on alliances in line graphs, see \cite{SigRod2006b}.
 This observation could motivate further studies in that direction.
\item Again in domination theory, dominating sets with additional properties like connectedness or independence have been
thoroughly studied. In the context of alliances, we are only aware of one such paper, namely~\cite{Vol2011}.
In particular those connected alliances should make very much sense in view of the original motivation of alliances. {$\frac{1}{2}$-independent dominating sets}
which correspond to independent global offensive $0$-alliances were studied in \cite{DahRauVol2008}.
These combinations could be also interesting for those interested in signed domination or monopolies.
Here, we would like to draw the reader's attention to \cite{DinNguTha2013}, where connected positive influence dominating sets are studied, which correspond to
connected global offensive $0$-alliances according to our observations.
\item While classical NP-hardness results for alliance problems are quite abundant and also follow often from other known results by connections exhibited in this paper,
positive algorithmic results are studied less. Some fixed-parameter algorithms have been exhibited, as indicated throughout the paper, but approximability is an issue largely neglected.
More precisely, while some results can be deduced by our results shown in this paper, a systematic research is still to be done.
\end{itemize}

\end{document}